\documentclass[12pt]{amsart}
\usepackage{amsthm}
\usepackage{cancel}
\usepackage{hyperref}
\usepackage[all]{xy}
\usepackage{a4wide}
\usepackage{array, xcolor}
\usepackage{longtable}
\usepackage{amsfonts}
\usepackage{amsmath}
\usepackage{mathtools}
\usepackage{amssymb}
\usepackage{graphicx}
\usepackage{accents}
\usepackage{verbatim}
\usepackage{times}
\newtheorem{proposition}{Proposition}
\newtheorem{lemma}[proposition]{Lemma}

\newtheorem{theorem}[proposition]{Theorem}
\newtheorem*{theoremut}{Uniformization theorem}
\newtheorem*{theoremrmt}{Riemann mapping theorem}
\newtheorem*{theorempit}{Poisson integral formula}

\theoremstyle{definition}
\newtheorem{definition}[proposition]{Definition}
\newtheorem{remark}[proposition]{Remark}
\newtheorem*{remark*}{Remark}

\def\RR{{\mathbb R}}

\def\NN{{\mathbb N}}
\def\CC{{\mathbb C}}

\def\DD{{\mathbb{D}}}

\newcommand{\supp}{\operatorname{supp}}

\begin{document}
\title{Uniformization of Riemann surfaces revisited}
\author{Cipriana Anghel}
\address{Institute of Mathematics of the Romanian Academy\\
Calea Grivi\c tei 21\\
010702 Bucharest\\ 
Romania}
\author{Rare\c s Stan}
\date{\today}
\begin{abstract}
We give an elementary and self-contained proof of the uniformization theorem for non-compact simply-connected Riemann surfaces.
\end{abstract}
\maketitle

\section{Introduction} 

Paul Koebe and shortly thereafter Henri Poincar\'e are credited with having proved in 1907 the famous \emph{uniformization theorem} for Riemann surfaces, arguably the single most important result in the whole theory of analytic functions of one complex variable. This theorem generated connections between different areas and lead to the development of new fields of mathematics.
After Koebe, many proofs of the uniformization theorem were proposed, all of them relying on a large body of topological and analytical prerequisites. Modern authors \cite{fkra}, \cite{forster} use sheaf cohomology, the Runge approximation theorem, elliptic regularity for the Laplacian, and rather strong results about the vanishing of the first cohomology group of noncompact surfaces. A more recent proof with analytic flavour appears in Donaldson \cite{don}, again relying on many strong results, including the Riemann-Roch theorem, the topological classification of compact surfaces, Dolbeault cohomology and the Hodge decomposition. In fact, one can hardly find in the literature a self-contained proof of the uniformization theorem of reasonable length and complexity. Our goal here is to give such a minimalistic proof. 

Recall that
a \emph{Riemann surface} is a connected complex manifold of dimension $1$, i.e., a connected 
Hausdorff topological space locally homeomorphic to $\CC$, endowed with a holomorphic atlas.

\begin{theoremut}[Koebe \cite{koebe}, Poincar\'e \cite{poincare}]
Any simply-connected Riemann surface is biholomorphic to either
the complex plane $\mathbb C$, the open unit disk $\mathbb D$, or the Riemann sphere $\hat{\mathbb{C}}$.
\end{theoremut}
A weaker result for domains in $\mathbb{C}$ appeared for the first time in Riemann's thesis \cite{riemann}: 
\begin{theoremrmt}
A simply-connected domain strictly contained in $\mathbb C$ is biholomorphic  to the unit disk.
\end{theoremrmt}
Riemann's proof for his mapping theorem was notoriously imprecise, as none of the necessary tools were yet fully developed at the time; Osgood \cite{osg} is credited with completing the argument, building on the work of many other mathematicians. 

What we show below is the following statement:
\begin{theorem}\label{inj}
Every non-compact simply-connected Riemann surface is biholomorphic to a domain in $\mathbb{C}$.
\end{theorem}
Combined with the Riemann mapping theorem,  
Theorem \ref{inj} is clearly equivalent to the non-compact case of the uniformization theorem.

The initial proofs of the uniformization theorem implicitly assumed the surface to have a countable basis of topology, but this hypothesis was shown to be unnecessary some 20 years later \cite{rado}. Most known proofs start from Riemann's idea of searching for so-called \emph{Green functions}, i.e., harmonic functions with a prescribed logarithmic singularity on sub-surfaces with boundary of a given Riemann surface. Such a Green function will turn out to be the log-norm of a holomorphic map. By exhausting the surface with compact sub-surfaces with boundary, one can find in principle a sequence of embeddings, and compactness results should yield a global embedding into $\mathbb{C}$. Technical complications arise while carrying out this program, and our task is to solve them with the minimum amount of machinery. 

Our paper starts with a self-contained treatment of a few necessary results from Perron's theory \cite{perron}
of subharmonic functions, deducing from there Rad\'o's theorem on the second countability of Riemann surfaces. We follow closely the line of proof of \cite{Hub} of constructing an exhaustion by compact surfaces with boundary by means of Sard's lemma. Hubbard \cite{Hub} proves the \emph{a priori} stronger statement that every non-compact Riemann surface with trivial first Betti number is biholomorphic to either $\mathbb{C}$ or the disk. 
One technical issue in \cite{Hub} is constructing the harmonic dual of the Green function and showing its continuity to the boundary, a particular case of the Osgood-Carath\'eodory theorem \cite{caratheod}, which he uses without proof. He must also rely on Koebe's $1/4$ theorem, further complicating the argument. 

The novelty in our approach is to go around these problems by yet another application of Sard's lemma, together with the standard Riemann mapping theorem. We also give a direct proof of the removal of singularities property for bounded harmonic functions, as we were unable to find a clear reference in the literature. Besides these new arguments, we also give full proofs of the necessary auxiliary results in the most economical way possible. 

Our prerequisites are exclusively contained in the standard undergraduate curriculum:
\begin{itemize}
\item the Schwarz lemma, Montel's theorem on normal families of holomorphic functions, and Riemann's mapping theorem for simply-connected plane domains; we refer to \cite{stein} or \cite{cartan};
\item the Poisson integral formula on the unit disk (see \cite[Ch.\ 22]{forster});
\item partitions of unity associated to open covers of second-countable manifolds, integration of $1$-forms along a smooth path, and Stokes formula in the plane; see \cite{tu};
\item Sard's lemma; see for instance \cite{milnor};
\item the Tietze extension theorem, the universal cover and the fundamental group $\pi_1$; see \cite{mccleary}.
\end{itemize}

What we \emph{do not} use here includes: the Riemann-Roch theorem, the $\partial\overline{\partial}$ lemma, sheaves and their cohomology, singular and de Rham cohomology and their isomorphism, Runge theory, Weierstrass' results on the existence of meromorphic functions, Koebe-Bieberbach's $1/4$ theorem, the existence of triangulations on surfaces, or the topological classification of compact surfaces. We hope that in this way, the uniformization theorem becomes accessible to a general audience, including not only students but also working mathematicians wishing to gain access to an elementary, short, and self-contained proof. 

The proof of the uniformization theorem is useful for the study of moduli spaces of hyperbolic metrics of families of non-compact, non simply-connected Riemann surfaces with prescribed geometry near infinity \cite{punctured_surfaces}, \cite{ricci_flow}. We believe that our approach can be similarly applied to the Schottky uniformization of compact Riemann surfaces, this will be carried out elsewhere.

The reader is referred to  \cite{PPP} for a comprehensive overview of problems in Riemann surface theory leading to the uniformization theorem, and to the collective work \cite{SG} for a historical account of the theorem and an outline of the proof.

\subsection*{Acknowledgements}
We would like to thank Sergiu Moroianu for useful discussions. The authors were
partially supported from the project PN-III-P4-ID-PCE-2020-0794 and from a Bitdefender Junior Scholarship.

\section{Local behaviour of holomorphic maps}\label{localbeh}
We will repeatedly use the local description of holomorphic maps as monomials in suitable charts.
Let $f:X \longrightarrow Y$ be a non-constant 
holomorphic map between Riemann surfaces.
Let $x_0\in X$. Choose charts $(\eta,U)$, $(\psi,V)$ 
centered at $x_0$ and $y_0=f(x_0)$ respectively. Then
$\psi \circ f \circ \eta^{-1}(z)=z^k u(z)$ for some $k\in\mathbb{N}$ and some holomorphic function $u$, with $u(0)=c\neq 0$. Choose $z_0\neq 0$. Let $v(z)=
e^{\frac{1}{k}\int_{z_0}^z\frac{du}{u}}$. Then $v(z)^k=\frac{u(z)}{u(z_0)}$. In the new chart on $X$ defined by
$\varphi(p)=\eta(p)v(\eta(p))$, 
we have 
$\psi \circ f \circ \varphi^{-1}(z)=cz^k$. In particular, it follows that non-constant holomorphic maps are \emph{open}: they map open sets in $X$ 
to open sets in $Y$.

\section{Perron's method} 

A \emph{ disk centered at $x$} in a Riemann surface $X$ is the preimage of the unit disk through a holomorphic chart $\varphi:U\longrightarrow \CC$ mapping $x$ to $0$, with $\overline{\DD}\subset \varphi(U)$.

\begin{definition}
Let $X$ be a Riemann surface. A continuous function $u:X\longrightarrow \RR$ is called
\emph{subharmonic}, respectively \emph{harmonic}, if for every disk in $X$, the real number:
\[
\frac{1}{2\pi}\int_{0}^{2\pi}u(e^{it})dt -u(0)
\]
is non-negative, respectively zero.
\end{definition}

\begin{theorempit}\label{dirichlet}
Every continuous function $f:S^1\longrightarrow \mathbb{R}$ admits a unique continuous extension $u:\overline{\mathbb{D}}\longrightarrow \mathbb{R}$ which is  harmonic on $\DD$. Moreover, on $\mathbb{D}$, $u$ is
given by the formula:
\[
u(z)= \frac{1}{2\pi} \Re \left( \int_{0}^{2\pi} f( e^{i\theta} ) 
\frac{e^{i\theta}+z}{e^{i\theta}-z}  d\theta \right) =\frac{1}{2\pi}
\int_{0}^{2\pi} f( e^{i\theta} ) \frac{1-|z|^2}{|e^{i\theta}-z|^2} d\theta.
\] 
\end{theorempit}
The proof is elementary, see for instance \cite[Theorem 22.3]{forster}. Since $\frac{e^{i\theta}+z}{e^{i\theta}-z}$ is holomorphic in $z$, every harmonic function is the real part of a holomorphic function, hence it is real analytic and satisfies the Laplace equation $\Delta u=0$. Conversely,  by Cauchy's integral formula, the real part of any holomorphic function is harmonic, so in particular $\log |\xi|=\Re \left( \log \xi \right)$ is harmonic.

\begin{proposition} \label{defechiv}
Let $u$ be a continuous function on a Riemann surface $X$. The following
definitions are equivalent:
\begin{itemize}
\item[D1: ] $u$ is subharmonic;
\item[D2: ] \emph{(The maximum principle)} For any harmonic function $h$ on any open connected subset $U\subset X$, either $u+h$ is constant on $U$, or it does not attain its supremum;
\item[D3: ]For every disk $D\subset X$ one has $u \leq u^{(D)}$, where $u^{(D)}$ is the unique continuous function
defined by modifying $u$ inside the disk by using the Poisson integral formula:
\begin{equation}\label{ud}
u^{(D)}(x)= 
\begin{cases}
u(x) &  x \in X\setminus D;\\          
        \text{harmonic} & x\in D.
\end{cases}
\end{equation}
\end{itemize}
\end{proposition}

We leave to the reader the (immediate) proof of this proposition. Remark that by the maximum principle, a non-constant subharmonic function cannot attain its supremum. Thus, the restriction of a subharmonic function to a compact set attains its maximum on the boundary.

\begin{remark}\label{sublocal}
Definition D2 is local: a function
which is subharmonic in the sense of D2 on a neighbourhood of every point is globally subharmonic.
We deduce that
the function $u^{(D)}$ constructed in \eqref{ud} is subharmonic on $X$. 
Indeed, $u^{(D)}$ is clearly subharmonic on $D$ and on $X \setminus
\overline {D}$.
Let $p$ be a point in $\partial D$. For every disk centered at $p$ we get:
\[
u^{(D)}(p)=u(p)\leq \frac{1}{2\pi} \int_0^{2\pi} u(e^{i\theta})d\theta \leq
\frac{1}{2\pi} \int_0^{2\pi} u^{(D)}(e^{i\theta})d\theta.
\]
Hence $u^{(D)}$ is subharmonic in a neighbourhood of $p$ as well, therefore  subharmonic on $X$.
\end{remark}

\begin{definition}
Let X be a Riemann surface. A set $\mathcal F$ of subharmonic functions on X is
called a \emph{Perron family} if the following two properties are verified:
\begin{itemize}
\item If $f,g \in \mathcal F$, then $\max(f,g) \in \mathcal F$.
\item If $f \in \mathcal F$, $D$ is a disk, and
$f^{(D)}$ is the subharmonic function defined by \eqref{ud}, then $f^{(D)} \in \mathcal F$.  
\end{itemize}
\end{definition}

\begin{theorem}\label{perron}
Let $X$ be a (possibly not second countable) Riemann surface.
\begin{itemize}
\item \emph{(Perron's principle)}
If $\mathcal F$ is a non-empty, locally bounded above Perron family 
on $X$, then $\sup_{f \in \mathcal F} f$ is harmonic. 
\item \emph{(Dirichlet's principle)} Let $Y \subset X$ be a sub-manifold
with (possibly non-compact) boundary and $m,M\in \RR$. If $f:\partial Y \longrightarrow [m,M]$ 
is a bounded continuous
function, then there exists a continuous function $\tilde{f}:Y \longrightarrow [m,M]$ which
extends $f$, and is harmonic on $\accentset{\circ}{Y}$.  
\end{itemize}
\end{theorem}

We prove this theorem in Appendix \ref{AA1} and \ref{AA2}. As a corollary we get:
\begin{theorem}[Rad\'o's theorem] \label{rado}
Every Riemann surface $X$ is second countable.
\end{theorem}
\begin{proof}
Let $\overline{D},\overline{D'}$ be two disjoint
closed disks in $X$. Consider $Y=X \setminus \left(
{D} \cup {D'} \right)$. By Dirichlet's principle 
(which does not require second countability), 
there exists a continuous 
function $g:Y\to \mathbb R$, harmonic on $\accentset{\circ}{Y}$ with $g_{\vert_{\partial D}}\equiv 0$, and 
$g_{\vert_{\partial D'}}\equiv 1$.
Since $g$ is harmonic, the $1$-form $\partial g=\frac{1}{2}\left(
\frac{\partial g}{\partial x}-i\frac{\partial g}{\partial y} \right)dz $ 
is holomorphic. Fix $y_0 \in \accentset{\circ}{Y}$. For every path $\gamma:[0,1] \longrightarrow \accentset{\circ}{Y}$
with $\gamma(0)=y_0$, let $G(\gamma)=\int_{\gamma} \partial g$. Then $G$ is a well-defined, non-constant, holomorphic
function on the universal cover of $\accentset{\circ}{Y}$, which by Appendix \ref{AB} must be second countable. This implies
immediately that both $\accentset{\circ}{Y}$ and $X$ are second countable.
\end{proof}

\section{Green functions}

\begin{lemma}\label{rsh}
Every harmonic function $f:\mathbb{D}^{*}\longrightarrow \mathbb{R}$ bounded
near $0$ extends continuously to a harmonic function on
$\mathbb{D}$.
\end{lemma}
\begin{proof}
Recall that from the Poisson integral formula $f$ is smooth, so we can define the 1-form $\beta=\frac{\partial f}{\partial x}dy-\frac{\partial f}{\partial y}dx$. Since $f$ is harmonic, it
follows that $\beta$ is closed. Choose $z_0\in \mathbb D^*$. 
The multi-valued map 
\begin{align*}
&g: \mathbb D^* \longrightarrow \mathbb R; &g(z)=\int_{z_0}^{z}  \beta,
\end{align*}  
defined by integration along a path from $z_0$ to $z$, depends
only on the homotopy class of the chosen path. 
Let $\gamma$ be the circle centered at $0$ passing through $z_0$. It is a generator of $\pi_1(\mathbb D^*)$, so
the map $g$ is well-defined modulo $a\mathbb Z$, where $a=\int_{\gamma} \beta$. Hence $f+ig$ is well-defined modulo $ia \mathbb Z$. We claim
that $f+ig$ satisfies the Cauchy-Riemann conditions. Indeed,
\begin{align*}
\frac{\partial g}{\partial y}(z)= \left. \frac{d}{dt} \right|_{t=0}
\int_{z}^{z+it}
\left( \frac{\partial f}{\partial x}dy-\frac{\partial f}{\partial y}dx  \right)
= \left. \frac{d}{dt} \right|_{t=0} \int_{0}^{t} \frac{\partial f}{\partial
x} (z+i \theta) d\theta= \frac{\partial f}{\partial x} (z). 
\end{align*}
The other condition is verified similarly. If $a=0$, $e^{f+ig}$ is a
well-defined holomorphic function on $\mathbb
D^*$. Moreover, $e^{f+ig}$ is bounded near 0, thus its singularity in $0$ is removable. Since $\mathbb D$ is simply-connected and $\left| e^{f+ig}
\right|=e^{f} >0$, we deduce that $f=\Re(\log e^{f+ig})$ extends
harmonically on $\DD$. Otherwise, if $a\neq 0$, we apply the same argument to the function $e^{\frac{2 \pi }{a}(f+ig)}$.
\end{proof}

\begin{proposition}\label{green}
Let $K\subset X$ be a (not necesarily compact) sub-surface with smooth boundary and fix $x_0\in \accentset{\circ}{K}$. Then there exists a continuous function $G:K \setminus
\lbrace x_0 \rbrace \longrightarrow [0,\infty)$ which is harmonic on the
interior of $K \setminus \lbrace x_0 \rbrace$, vanishes on $\partial K$, and
has a logarithmic pole in $x_0$ (i.e., $G+ \log|\xi|$ is harmonic on a
neighbourhood of $0$, where $\xi$ is a local holomorphic coordinate centered at $x_0$).
\end{proposition}
\begin{proof}
Consider a disk $D$ centered at $x_0$ with coordinate $ \xi$, and let $D_{\frac{1}{2}}=\{|\xi|<\frac{1}{2}\} \subset D$. From Dirichlet's principle we can define a continuous map $h_1:K \setminus D_{\frac{1}{2}} \longrightarrow [0,1]$ by:
\[
h_1(\xi)= \begin{cases}
        1 &  |\xi|=\frac{1}{2}  \\
        0 &  \xi \in \partial K \\          
        \text{harmonic,} & \mbox{otherwise.}
\end{cases} \]
By the maximum principle, $a= \sup_{\partial D} h_1 \in (0,1)$ so we can take $A,B > 0$ such that: 
\[
 Ba < A < B- \log 2.
\]
From these inequalities, it follows that the map:
\begin{align*}
h:\overline{D} \setminus D_{\frac{1}{2}}\longrightarrow \mathbb R, &&h(\xi)=\max(-Bh_1(\xi),\log|\xi|-A)
\end{align*}
can be continuously extended to a subharmonic function on $K \setminus \{ x_0\}$, by setting it to be $-Bh_1$ on $K \setminus D$, and $\log|\xi|-A$ on $D_{\frac{1}{2}}$.

Let $\mathcal F$ be the Perron family of those continuous functions $g:K
\setminus \lbrace x_0 \rbrace \longrightarrow [0,\infty)$ which are subharmonic on the interior,
restrict to $0$ on $\partial K$, and satisfy $g+h \leq 0$. The family $\mathcal F$ is bounded above (by $-h$) and contains the map $\alpha: D \setminus \lbrace 0 \rbrace \longrightarrow \mathbb R$,
$\alpha(\xi)=-\log|\xi|$ extended with $0$ outside the unit disk. By Perron's principle, the map $G:K\setminus\{x_0 \}\longrightarrow [0,\infty)$, $G=\sup_{g\in\mathcal{F}} g$ is continuous, vanishes on $\partial K$ and is harmonic on $\accentset{\circ}{K} \setminus \{x_0\}$. Moreover $\alpha \leq G \leq -h$, so $-\log|\xi| \leq G \leq A-\log|\xi|$ on $D_{\frac{1}{2}}$ and therefore $G+\log|\xi|$ is bounded near $0$. Lemma \ref{rsh} shows that $G+ \log|\xi|$ is actually harmonic near $0$, ending the proof.
\end{proof}

\section{A holomorphic embedding in the unit disk} 

\begin{theorem}\label{olomxn}
Let $X$ be a Riemann surface and $K\subset X$ a compact simply-connected sub-surface with precisely one (smooth) boundary component. Then there exists a biholomorphism $\varphi:\accentset{\circ}{K}
\longrightarrow  \DD$.
\end{theorem}
\begin{proof}
Let $G$ be the Green function from Proposition \ref{green} and consider the $1$-form given in holomorphic coordinates by $\omega=\frac{\partial G}{\partial x}dy-\frac{\partial G}{\partial y}dx \in \Omega^1(\accentset{\circ}{K})$. One can easily check that it is invariant under holomorphic changes of variables (in fact, this form is just $JdG$, where $J$ is the almost complex structure).  
Fix $z_0$ in $\accentset{\circ}{K} \setminus \lbrace x_0 \rbrace$ and
consider the multi-valued map 
\begin{align*}
&F:\accentset{\circ}{K} \setminus \lbrace x_0 \rbrace \longrightarrow \mathbb R; 
&F(z)= \int_{z_0}^{z} \omega 
\end{align*}
defined by integrating along any smooth path from $z_0$ to $z$. As above,
the value of $F(z)$ depends only on the homotopy class of the chosen path. We claim that $F(z)$ is well-defined modulo $2\pi$.
\begin{remark}
One could try to prove that $\pi_1(\accentset{\circ}{K} \setminus \lbrace x_0 \rbrace)$ is cyclic as an application of Van Kampen's theorem, since
$\accentset{\circ}{K} \setminus \lbrace x_0 \rbrace$ is obtained from $\accentset{\circ}{K}$ by removing a disk. In reality, all we can say is that $\pi_1(\accentset{\circ}{K} \setminus \lbrace x_0 \rbrace)$ equals the normalizer of a cyclic subgroup. We will avoid this problem by using the Sard lemma.
\end{remark}

Take $\gamma:S^1\longrightarrow \accentset{\circ}{K}\setminus\{x_0\}$ a smooth loop with $\gamma(1)=z_0$. Since $\accentset{\circ}{K}$ is simply-connected there exists $h:\overline{\DD}\longrightarrow \accentset{\circ}{K}$, a smooth extension of $\gamma$ to the closed unit disk. 

We can assume $x_0$ to be a regular value for $h$. If not, consider a disk $D$ centered at $x_0$ which does not intersect $\gamma$. By Sard's lemma, there exists a regular value $y_0$ for $h$ inside $D$. Pick $ \psi:\accentset{\circ}{K}\longrightarrow \accentset{\circ}{K}$ a diffeomorphism which maps $y_0$ to $x_0$, and acts as the identity outside $D$. Then $\psi\circ h$ is a new homotopy for which $x_0$ is a regular value. 

Hence $h^{-1}(x_0)$ is a (possibly empty) finite set in the disk $\DD$. By pulling back and using Stokes' theorem, we get:
\[
\int_{\gamma} \omega =\int_{S^1}h^*\omega=\sum_{q\in h^{-1}(x_0)} \int_{\mathcal{C}_q}h^*\omega=\sum_{q\in h^{-1}(x_0)} \int_{h(\mathcal{C}_q)}\omega,
\]
where the $\mathcal{C}_q$'s are disjoint circles centered at $q$, for each $q\in h^{-1}(x_0)$, chosen such that their image through $h$ lies inside a punctured disk $D\setminus\{x_0\}$. Every circle $\mathcal{C}_q$ is homotopic in $D\setminus \{x_0\}$ to the circle $|\xi|=r$, where $\xi$ is a local coordinate in a disk centered at $x_0$ and $r<1$ is a positive number. Therefore the right-hand side is an integer multiple of $\int_{|\xi|=r} \omega$. 
From Proposition
\ref{green}, we know that $H:=G+ \log |\xi|$ is harmonic on
$\{ |\xi| < 1 \}$, hence, using Stokes' theorem:
\begin{align*}
\int_{ |\xi|=r}  \frac{\partial G}{\partial x}dy-\frac{\partial G}{\partial
y}dx &=\int_{|\xi|\leq r} \left( \frac{\partial^2 H}{\partial
x^2}+\frac{\partial^2H}{\partial y^2}\right) dx\wedge dy -\int_{|\xi|=r}
\frac{xdy-ydx}{x^2+y^2}=-2\pi.
\end{align*}
This proves the claim that $F$ is well-defined modulo $2\pi$, therefore
$\varphi=e^{-G-iF}:\accentset{\circ}{K}\setminus \{x_0\}\longrightarrow \DD$ is holomorphic and well-defined. The singularity in $0$ of such a function is removable because $|e^{-G}|=|xi|e^H$ is bounded, so we can extend $\varphi: \accentset{\circ}{K} \longrightarrow \mathbb D$ with $\varphi(x_0)=0$. Clearly,
$\varphi^{-1}(0)=\{x_0\}$.

Let us prove that $\varphi$ is injective. Let $A$
be the set of those points $x \in
\accentset{\circ}{K}$ for which $d\varphi_x \neq 0$ and the preimage $\varphi^{-1}(\varphi(x))$ only contains $x$. First we prove
that $A$ is open. Let $x \in A$. By continuity and using Section \ref{localbeh} we find a disk
$D$ centered at $x$ where $d\varphi \neq 0$ and on which $\varphi$ is injective. Suppose there
existed
sequences $(x_j)_{j \in \mathbb N}$, $(y_j)_{j \in
\mathbb N} \subset \accentset{\circ}{K}$ with $x_j \neq y_j$, $x_j
\rightarrow x$ and
$\varphi(x_j)=\varphi(y_j)$. Since $\varphi$ is injective on
$D$, there exists a rank $j_0$ such that for any $j \geq j_0$,
$x_j \in D$ and $y_j \notin D$. Since
$(y_j)_{j \geq j_0} \subset K \setminus D$,
there exist $ y \in K\setminus D$ and a subsequence $(y_{j_k})_{k \in
\mathbb N}$ such that $\lim_{k\rightarrow \infty} y_{j_k}=y$. If $y\in \accentset{\circ}{K}$ then we have $\varphi(x)=\varphi(y)$, contradicting $x\in A$. Otherwise, if $y\in \partial K$, we have $|\varphi(y_{j_k})|\to e^{-G(y)}=1$ while $|\varphi(x_{j_k})|\to |\varphi(x)|<1$, contradiction.

Now we prove that $A$ is closed in $\accentset{\circ}{K}$. Let $(x_j)_{j \in
\mathbb N} \subset A$, with $x_j \rightarrow p \in \accentset{\circ}{K}$. If
$\varphi'(p)=0$, then again by Section \ref{localbeh}, in a suitable chart $U$ near $p$, the map $\varphi$ is given
by $z\mapsto cz^k$, for some positive integer $k\geq 2$. Let $\zeta\neq 1$ be a $k$-th
root of unity. Then, for $j$ large enough, we have that $\varphi(x_j)=\varphi(\zeta x_j)$, which is a contradiction. Otherwise,
suppose there existed another point $q\in \accentset{\circ}{K}$ with
$\varphi(p)=\varphi(q)$. Since $\varphi$ is an open map, we can find two
disjoint neighborhoods $U\ni p$ and $V\ni q$ with $\varphi(U)=\varphi(V)$.
But this means that, for $j$ large enough, $\varphi^{-1}(x_j)\cap V \neq
\emptyset$, which is again a contradiction. Hence $A$ is closed in
$\accentset{\circ}{K}$.

To conclude that $A=\accentset{\circ}{K}$, we are left to prove that $A\neq
\emptyset$. Note that near $x_0$ we have
$|\varphi(\xi)|=|e^{\log|\xi|-H(\xi)}|=|\xi||e^{-H(\xi)}|$.
Thus $\varphi'(x_0)\neq 0$, so $x_0\in A$.

Since non-constant holomorphic maps are open, the injection $\varphi$ is a biholomorphism onto its image $\varphi(\accentset{\circ}{K})\subset \DD$. 

The Riemann mapping theorem allows us to conclude that $\accentset{\circ}{K}$ is biholomorphic to $\DD$. 
\end{proof}

\section{Exhaustion by compact simply-connected submanifolds}
Let $X$ be a simply-connected, non-compact Riemann surface and $x_0 \in X$ a
fixed point.
\begin{theorem}
There exists a compact exhaustion of $X$
\begin{align*}x_0 \in K_0 \subset K_1 \subset ... \subset X,&&K_n\subset
\accentset{\circ}{K}_{n+1},&&\bigcup_{n\geq 0}K_n=X
\end{align*}
with simply-connected
surfaces with non-empty connected smooth boundary.
\end{theorem}
\begin{proof}
By Rad\'o's theorem we can choose a countable locally finite cover of $X$ by relatively compact sets
$(U_m)_{m\in \mathbb{N}}$
and a subordinated partition of unity $(\phi_m)_{m\in \mathbb{N}}$. Define a smooth function
\begin{align*}
g:X\longrightarrow [0,\infty),&& 
g(x)=\sum_{m\in \mathbb{N}} m \phi_m (x).
\end{align*}
For every integer $m$, $g^{-1}([0,m]) \subset \supp \phi_1 \cup...\cup \supp
\phi_{m}$ is compact, hence
$g$ is proper and unbounded above.
By Sard's lemma, there exists an increasing sequence of regular values
$(a_n)_{n \in \mathbb N}$ with $a_n\rightarrow \infty$ and we may suppose that
$a_0>g(x_0)$.
Consider $Y_n'=\lbrace x \in X \vert g(x)\leq a_n \rbrace $ and let $Y_n$ be the
connected component of $Y_n'$ containing $x_0$. Obviously $Y_n' \subset
Y_{n+1}'$, so $Y_n \subset Y_{n+1}$.
Since $X$ is a connected manifold, it is path-connected 
so, for every $x\in X$, there exists a path $\gamma:
[0,1]\longrightarrow X$ with $\gamma(0)=x_0$ and $\gamma(1)=x$. The set
$g(\gamma[0,1])$ is compact, so 
$g(\gamma[0,1]) \subset [0,a_k]$ for some $k \in \mathbb N$. 
Then $\gamma([0,1]) \subset Y_k'$ and, by connectedness,
$\gamma([0,1]) \subset Y_k$, hence $x\in Y_k$ so $ \bigcup_{n} Y_n=X$. 

The compact sub-manifolds $Y_n$ are not exactly what we are looking for because they might have 'holes', like an annulus in $\CC$ (they are not holomorphically convex). The natural way to fill these holes is to add those connected components of $X \setminus Y_n$ with compact closure. Since any two components have distinct boundaries and the boundaries are subsets of $g^{-1}(a_n)$ (which is a compact $1$-manifold) we can define the compact sub-manifold $K_n$ as the finite union of $Y_n$ with these connected components. We can easily convince ourselves that $\bigcup_{n \in \mathbb N} K_n=X$, $K_n \subset \accentset{\circ}{K}_{n+1}$, and $K_n$ is connected (see picture).
\begin{center}
\includegraphics[width=12cm, height=4cm]{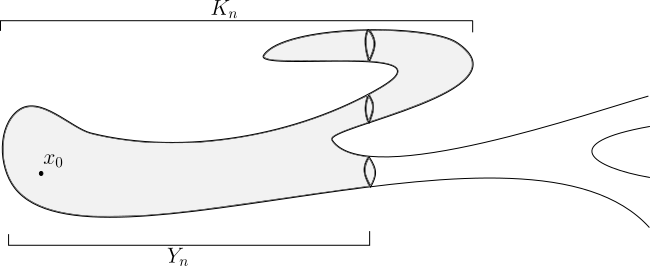}
\end{center} 
To finish the proof we have to show that $\pi_1(K_n,x_0)=0$. The idea is to find a retraction from $X$ to $K_n$. Let $Z$ be the closure of a connected component of
$X \setminus K_n$. First, we need to show that $\partial Z$ is connected.
Assume by contradiction that there existed two components $\gamma_1$ and $\gamma_2$ of $\partial Z$. We take a
point on each and connect them by smooth simple paths $\delta_1$ in $Z$ and $\delta_2$
in $K_n$, both transversal to $\gamma_1$ and $\gamma_2$ (the existence of such paths linking any two points on a
Riemann surface can be easily proved: the set
of points which can be connected by a simple curve to a given point is both
open and closed). Moreover we can assume that $\delta=\delta_1 \cup \delta_2$ is smooth. Then $\delta$ is
diffeomorphic to $S^1$ and we can consider a tubular neighbourhood $S^1\times (-\epsilon,\epsilon) \subset X$ with
coordinates $(x,y)$, where $x$ is the coordinate on the circle. Choose $\eta :
\mathbb R \longrightarrow \mathbb R$ a test function such that $\supp \eta \subset (-\epsilon,\epsilon)$ and $\int_{-\infty}^{\infty} \eta(y) dy=1$. Define a closed $1$-form $\omega \in
\Omega ^1(X)$ by $\omega = \eta (y) dy$, extended by $0$ outside the tubular neighbourhood. From the homotopy invariance of the integral, we can assume that the intersection of $\gamma_1$ with the tubular neighbourhood is a segment of type $\{x=\text{constant} \}$, hence $\int_{\gamma_1}\omega=\pm 1$. However $X$
is simply-connected, therefore the integral of any closed $1$-form on a loop is $0$, yielding the desired contradiction. Thus
$\partial Z$ is connected, and we parametrize it by $\gamma :[0,1] \longrightarrow
\partial Z \subset \partial K_n$, where $\gamma (0)=\gamma (1)=p$. 

Let $\delta$ be a simple path in $Z$ starting at $p$ escaping from every compact subset of $Z$. We cut $Z$ along $\delta$ to obtain a topological surface $\tilde{Z}$
with connected boundary $\partial \tilde{Z} = \delta' \cup [p',p''] \cup \delta''$ (see picture). 
\begin{center}
\includegraphics[width=12cm,height=4cm]{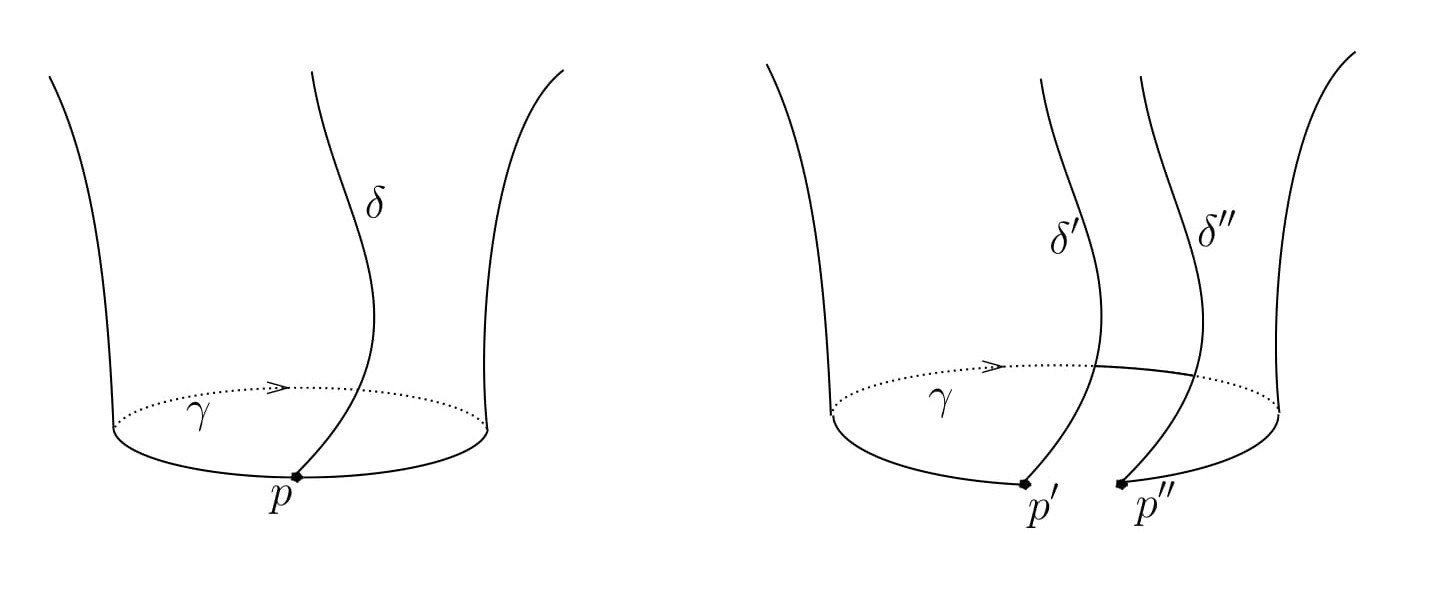}
\end{center} 
Let
$\mu: \partial \tilde{Z} \longrightarrow [0,1]$ be the continuous function 
$$
\mu(x)= \left\{
    \begin{array}{ll}
        0 & \mbox{ } x \in \delta' \\
        \gamma^{-1}(x) & \mbox{ } x \in \gamma \setminus \{ p \} \\          
        1 & \mbox{ } x \in \delta''.
    \end{array}
\right.
$$
By the Tietze extension theorem, we can extend $\mu$ to a continuous function
$\tilde{\mu}:\tilde{Z} \longrightarrow [0,1]$. Note that the composition $\gamma
\circ \tilde{\mu}$ takes the same value (namely $p$) on $\delta'$ and $\delta''$, therefore
it descends to a well defined map $\Psi_{Z} :Z \longrightarrow \gamma([0,1])\subset X$:

\[
\xymatrixcolsep{3pc}\xymatrix{
\tilde{Z}\ar[d] \ar[r]^{\tilde{\mu}}&[0,1]\ar[r]^{\gamma} & X\\
Z \ar@{-->}[urr]_{\Psi_{Z}} 
}
\]

Now let
$\rho_n:X\longrightarrow K_n$ be the retraction which acts as the identity on $K_n$ and
equals $\Psi_{Z}$ on every connected component $Z$ of $X \setminus
K_n$. Since the composition
 \[ K_n\hookrightarrow X \xrightarrow{\rho_n} K_n \] 
 is the identity map, it induces the identity on the fundamental groups. But
$\pi_1(X)=0$, thus $K_n$ is simply-connected. 
\end{proof}

\section{End of the proof of theorem \ref{inj}}
For $n \geq 0$, let $\phi_n:\accentset{\circ}{K_n} \longrightarrow r_n \DD$ be the unique
 biholomorphism given by Theorem \ref{olomxn} with $\phi_n(x_0)=0$ and $\phi_n'(x_0)=1$, where $r_n$ is a positive number determined by $K_n$,
 and the derivative is computed in a fixed coordinate near $x_0$.

For every Riemann surface $S$, let $\mathcal O(S)$ be the topological vector space of holomorphic functions from $S$ to $\CC$ with the topology defined by uniform convergence on compact subsets of $S$. 
\begin{lemma}\label{lema}
Let $(f_n:\DD \longrightarrow \mathbb C)_{n \in \mathbb N}$ be a sequence of \emph{injective} holomorphic functions with $f_n(0)=0$ and $f_n'(0)=1$. Then there exists a convergent subsequence $f_{n_k} \to f\in \mathcal O(\DD)$.
\end{lemma}
We defer the proof of this lemma to Appendix \ref{AC}.
Consider the holomorphic map $g_n: \mathbb D \longrightarrow \mathbb C$, $g_n(z)=\frac{1}{r_0}
\phi_n \circ \phi_0^{-1}(r_0z) $. Clearly $g_n$ is injective, $g_n(0)=0$, and $g_n'(0)=1$. 
By Lemma
\ref{lema}, there exists a convergent subsequence $(g_{n_{0k}})_{k \in \mathbb N}$. 
Therefore $(\phi_{n_{0k}})_{k \in \mathbb
N}$ converges on $\accentset{\circ}{K_0}$. Similarly, from the
sequence $\left( \frac{1}{r_1} \phi_{n_{0k}} \circ \phi_1^{-1}(r_1z)
\right)_{k\in \mathbb{N}}$, extract a subsequence $( \phi_{n_{1k}})_{k \in
\mathbb N}$ which converges on $\accentset{\circ}{K_1}$. 
Repeating this
process, we obtain the sequences $(\phi_{n_{mk}})_{k \in \mathbb N}$ which
converges on $\accentset{\circ}{K}_m$. Set $\phi_{n_k}=\phi_{n_{kk}}$. Now $(\phi_{n_k})_{k \in \mathbb N}$ converges on $X$ to a holomorphic function $\phi:X \longrightarrow {\mathbb C}$, uniformly on every compact of $X$. By continuity $\phi'(0)=1$, so $\phi$ is non-constant. Let us show that $\phi$ is injective: suppose there existed $z_1\neq z_2 \in X$ with $\phi(z_1)=\phi(z_2)$. Since $\phi$ is non-constant, there exists a circle $\mathcal C$ centered at $z_2$ whose interior does not contain $z_1$, and such that the difference $\phi-\phi(z_2)$ does not vanish on $\mathcal C$. Using the dominated convergence theorem, we get
\[
\lim_{n\rightarrow \infty} \frac{1}{2 \pi i} \int_{\mathcal C} \frac{\left( \phi_n(z)-\phi_n(z_1) \right)'(z)}{\left( \phi_n(z)-\phi_n(z_1) \right)(z)}dz= \frac{1}{2 \pi i} \int_{\mathcal C} \frac{(\phi-\phi(z_1)) '(z)}{\left( \phi-\phi(z_1) \right) (z)}dz, 
\]
The first integral is zero, while the second one counts the number of zeros of the function $\phi-\phi(z_1)$ in the disk bounded by $\mathcal C$, which is at least $1$, yielding a contradiction. Thus, $\phi$ is a biholomorphism onto
its image, ending the proof of Theorem \ref{inj}.

\appendix
\section{Proof of Perron's principle} \label{AA1}

Since the family $\mathcal{F}$ is locally bounded above, we can
define
$u:X \longrightarrow \mathbb{R}$, $u=\sup_{f\in \mathcal{F}} f.$
Our aim is to show that $u$ is harmonic. Since harmonicity is a local property, it suffices to prove that $u$ is harmonic on a
disk $D\subset X$. Let $A=\{ z_0, z_1,... \} \subset D$ be a dense subset. For each $j$, there exists a sequence $(v_{jk})_{k \in \mathbb N}
\subset \mathcal F$ such that $u(z_j)=\lim_{k\rightarrow \infty} v_{jk}(z_j)$.
Since $\mathcal F$ is a Perron family, the map $h_1=v_{11}^{(D)}$ belongs to $\mathcal F$, is harmonic on $D$, and $h_1 \geq v_{11}$. Suppose we constructed
the functions $\{h_1,...,h_n\} \subset \mathcal F$ harmonic on $D$ such that
$h_k \geq h_{k-1}$ and $h_k \geq v_{ij}$ for all $k\in \{1,...,n\}$, $i,j\in
\{1,...,k\}$. Pick $h_{n+1}=\max (v_{11},v_{12},...,h_n)^{(D)} \in \mathcal F$ harmonic on $D$, with $h_{n+1}
\geq h_{n}$ and $h_{n+1} \geq v_{ij}$ for any $i,j \in \{1,...,n+1\}$. Then
$h_n(z_j) \geq v_{jk}(z_j)$ for all $n \geq k \geq j$. Letting $k\rightarrow \infty$
in the previous inequality, we get: 
\[\lim_{n\rightarrow \infty}h_n(z_j)=u(z_j).\]     
Since $(h_n)_{n \in \mathbb N}$ is an increasing sequence, we can consider
the Lebesgue measurable function
\begin{align*}
w:D\longrightarrow \mathbb{R}, && w=\lim_{n\rightarrow \infty} h_n\leq 
\sup_{g\in\mathcal{F}}g_{\vert_{D}}=u_{\vert_D}.
\end{align*} 
By the dominated convergence theorem, for every disk $D'$ centered at a point $x\in D$ we have
\[
w(x)=\lim_{n\rightarrow \infty} h_n(x)=\lim_{n\rightarrow \infty} \frac{1}{\pi
}\int_{D'} h_n(z)dz=\frac{1}{\pi }\int_{D'} \lim_{n\rightarrow \infty}
h_n(z)dz=\frac{1}{\pi }\int_{D'} w(z)dz.
 \]
Here we used the mean property on disks for harmonic functions, which is an immediate consequence of the definition.
This implies easily that $w$ is $C^0$. Again by dominated convergence,
\[
 \frac{1}{2\pi} \int_{0}^{2\pi} w(e^{it}) dt=\frac{1}{2\pi} \int_{0}^{2\pi}
\lim_{n\rightarrow \infty} h_n(e^{it}) dt= \lim_{n\rightarrow \infty}
\frac{1}{2\pi}\int_{0}^{2\pi} h_n(e^{it}) dt=\lim_{n\rightarrow \infty}
h_n(x)=w(x)
 \]
thus showing that $w$ is harmonic. We
claim that $w=u$. 

Since $h_n \in \mathcal F$ for any $n \in \mathbb N$, then $h_n \leq u$, so $w
\leq u$. But $w(z_j)=u(z_j)$ for all $j \in \mathbb N$, hence $w \geq v$ on the 
dense subset $A \subset D$ for every $v \in \mathcal F$. Since $\mathcal F$ is a family of
continuous functions, it follows that $w \geq v$ for any $v \in \mathcal F$,
therefore $w \geq u$, which ends the proof.

\section{Proof of Dirichlet's principle} \label{AA2}

If $f$ is constant, the conclusion is clear, otherwise let $m<M$ be the infimum, respectively the supremum of $f$ on $\partial Y$.
Consider the family $\mathcal F$ of continuous functions $g:Y \longrightarrow [m,M]$
which are subharmonic on the interior of $Y$ with $g_{\vert_{\partial Y}}
\leq f$. 
The first condition for $\mathcal F$ to be a Perron family is evident.
Let $g \in \mathcal F$ and $D'$ a disk in $Y$. Using Remark \ref{sublocal}, $g^{(D')}$ is subharmonic. By the maximum principle, $g^{(D')}$ still takes values in $[m,M]$. Thus $g^{(D')} \in \mathcal F$. 
Now $\mathcal{F}\neq \emptyset$ since it contains the constant function
$m$. Let $F=\sup_{g \in \mathcal F} g$. By Perron's principle,
$F$ is harmonic on $\accentset{\circ}{Y}$. 

We must prove that for every $x\in\partial Y$, $\lim_{\xi \to x}F(\xi)=f(x)$. Let us first show
that $\liminf_{\xi\to x}
F(\xi)\geq f(x)$. If $f(x)=m$, there is nothing to prove, hence let us suppose that $f(x)>m$. We work in a disk $D\subset X$ centered at $x$ with local coordinate $\xi$.

Fix $\epsilon>0$ such that $f(0)-\epsilon>m$. There exists $1>\delta>0$ such that for every $\xi
\in\partial Y$
with $|\xi|<\delta$ we have $f(0)-\epsilon <f(\xi)$. Consider the exterior normal to $\partial Y$ at $x=0$. Take a circle centered at a point $p$ on this normal, of radius $r$ small enough such that it touches 
$Y$ only in $0$. Choose $t>1$ such that $\log\frac{1}{t-1}+f(0)-\epsilon<m$ and set $R=rt$. By decreasing $r$ if needed, we can assume that $R<\delta$. Clearly, $|p|\leq |\xi-p|$ for every $\xi\in Y\cap D$.
Define
\begin{align*}
u:D\to [m, \infty),&&
u(\xi)=\max \left(m,\log \tfrac{|p|}{|\xi-p|}
+f(0)-\epsilon\right).
\end{align*}
Notice that $u$ is the maximum of two harmonic functions, hence subharmonic. When $|\xi|$ is close to $0$,  
$u(\xi)=\log \tfrac{|p|}{|\xi-p|}+f(0)-\epsilon$, while for $|\xi|\geq R$, $u(\xi)=m$. Hence $u$ can be continuously extended (by the constant value $m$) to $Y\cup D$, and is subharmonic on $\accentset{\circ}{Y}$. It is straightforward to check that $u_{\vert_{Y}}$ belongs to the family $\mathcal{F}$. Thus, for small $|\xi|$, we get:
\begin{align*}
 F(\xi) \geq u(\xi)=\log\tfrac{|p|}{|\xi-p|}+f(0)-\epsilon.
\end{align*}
Since $\epsilon$ was arbitrarily small, we obtain $\liminf_{\xi \to 0} F(\xi) \geq f(0)$.
We now prove that $\limsup_{\xi \to 0} F(\xi) \leq f(0)$. If $f(0)=M$, this is clear. Otherwise, take $\epsilon>0$ such that $f(0)+\epsilon< M$. As above, consider $t>1$ and $R=rt$ such that $-\log \frac{1}{t-1}+f(0)+ \epsilon>M$, and $f(\xi)<f(0)+ \epsilon$ for $|\xi| \leq R$. Define 
\begin{align*}
U:D\to (-\infty,M],&&
U(\xi)=\min \left(M,-\log \tfrac{|p|}{|\xi-p|} + f(0)+\epsilon\right).
\end{align*}
Notice that $U$ is the minimum of two harmonic functions, thus $-U$ is subharmonic. Also $U(\xi)=M$ for
$|\xi|>R $, $U(\xi)=-\log \tfrac{|p|}{|\xi-p|} + f(0)+\epsilon$ for small $|\xi|$, $U \geq m$, and
$U(\xi)\geq f(\xi)$ on $\partial Y$. For every $g\in\mathcal{F}$, the
continuous function
$g-U$ is non-positive on the boundary of the compact domain $\{|\xi|\leq R \}\cap
Y$ and subharmonic in the interior.
By the maximum principle it follows that on this compact set $g\leq U$, hence
$F\leq U$. Thus for small $|\xi|$, 
\begin{align*}
F(\xi)\leq -\log\tfrac{|p|}{|\xi-p|}+f(0)+\epsilon.
\end{align*}
Therefore $\limsup_{\xi\to 0} F(\xi) \leq f(0)$, showing that $\lim_{\xi \to 0} F(\xi)$ exists and equals $f(0)$.

\section{Second countability in the presence of a holomorphic function}\label{AB}
Let us prove that if there exists a non-constant holomorphic function $f:X\longrightarrow \mathbb
C$, then the Riemann surface $X$ is second countable.

Let $\mathcal B$ be a countable basis of topology for $\mathbb C$. Let $\mathcal
A$ be the set of those connected components of each $f^{-1}(U)$, $U \in \mathcal
B$,
which are second countable. We claim that $A$ is a basis of topology for $X$.
Indeed, let $D \subset X$ be an open set and $x \in D$. By the identity theorem for holomorphic
functions, $f^{-1}\left (
f(x) \right)$ is discrete, hence there exists an open  neighbourhood
$W$ of $x$, relatively compact in $D$, such that $W \cap f^{-1}
\left( f(x) \right)= \lbrace x  \rbrace$. We have that $f \left( \partial \overline{W}
\right)$ is compact in $\mathbb C$. Since $f(x) \notin
f \left( \partial \overline{W} \right)$, there exists $U \in \mathcal B$ which contains
$f(x)$ such that $U \cap f \left( \partial \overline{W} \right) =\emptyset$. Let $V$ be the connected component of $f^{-1}(U)$ which contains $x$. Since $V \cap
\partial \overline{W}=\emptyset$, we get that $V \subset W$, hence $V$ is second countable.
We found $x \in V \subset D$, with $V \in \mathcal A$, therefore the claim is
proved. We next prove that $\mathcal A$ is countable.

Each $V \in \mathcal A$ intersects countably many other
elements in $\mathcal A$. Otherwise, there would exist $U \in \mathcal B$ such that
$V$ intersects uncountably many connected components of $f^{-1}(U)$. It would follow that $V$ contains uncountably many disjoint open subsets, which is a
contradiction.

Consider $V_0 \in \mathcal A$. There exists a countable number of open sets from
$\mathcal A$ which intersect $V_0$, denote their union with $V_0$ by $V_1$. For $k \geq 1$, define $V_{k+1}$ as the union of $V_k$ with those open sets in $\mathcal A$ which intersect $V_k$. By induction, $V_k$ is countable. It is clear that $\cup_{k \geq 1} V_k$ is both open and closed, hence $\cup_{k \geq 1} V_k=X$. Therefore all the open sets from $\mathcal A$ appear in this process, so $\mathcal A$ is a countable union of countable sets, hence countable.

\section{Compactness of families of injective holomorphic functions} \label{AC}

\begin{proof}[Proof of Lemma \ref{lema}]
Denote by $r_n=\sup \{ r \in
\mathbb{R}: r\DD \subset f_n(\mathbb D) \}$. Let $a_n \in \partial (r_n \overline { \DD}) \setminus f_n(\mathbb D)$. From the Schwarz lemma for the
function ${f_n^{-1}}_{\vert_{r_n \DD}}$, it
follows that $|a_n|=r_n \leq 1$. By extracting a subsequence, we can assume that $(a_n)_{n \in \NN}$ is convergent. Consider the
function $g_n: \mathbb D \longrightarrow \CC$,
$g_n=a_n^{-1} f_n$. It is easy to see that $\mathbb D \subset
g_n(\mathbb D)$ and also $1 \notin g_n(\mathbb D)$. Since $g_n$ is holomorphic
and injective, it follows that $g_n(\mathbb D)$ is simply-connected, hence we can construct the square root
\begin{align*}
&\psi_n:g_n(\mathbb D)\longrightarrow \mathbb C^{*}
&\psi_n(z)=ie^{\frac{1}{2}\int_0^z \frac{dw}{w-1}}.
\end{align*}
Then $\psi_n(0)=i$ and $\psi_n^2(z)=z-1$ for $z \in g_n(\mathbb D)$.
Since $\psi_n^2$ is injective, the image of $\psi_n$ cannot contain pairs of the form $(w,-w)$. Thus, there exists a disk $D$ centered at $-i$ disjoint from the image of $\psi_n$ for every $n$.
Define $h_n=\psi_n \circ g_n:
\mathbb D \longrightarrow \CC\setminus D$. There exists $0<r<\infty$ so that
the homography $\alpha(z)=\frac{1}{z+i}$ maps $\CC\setminus D$ into $r\DD$. Using
Montel's Theorem and relabeling the sequence, we can assume that $ \alpha \circ h_{n} $ converges to a
holomorphic map $h:\mathbb{D}\longrightarrow r  \overline{\DD}$. If
$h$ is non-constant, it is open by Section \ref{localbeh}, hence it maps $\mathbb{D}$ to $r\DD$.
Thus
we can compose it with the inverse homography $\alpha^{-1}$. Otherwise,
$h$ equals the constant $\frac{1}{2i}$. In both cases, $h_n$
converges to $\alpha^{-1} \circ h$, thus $f_{n}=a_{n}\left( 1+ h_n^2 \right)$ converges in $\mathcal O (\DD)$.
\end{proof}

\end{document}